\documentclass[12pt]{article}

\usepackage{tikz}
\usepackage{subfigure}
\usepackage[english]{babel}

\usepackage[center]{caption2}
\usepackage{amsfonts,amssymb,amsmath,latexsym,amsthm}
\usepackage{multirow}

\newtheorem{thm}{Theorem}

\newtheorem{lem}[thm]{Lemma}

\newtheorem{conj}[thm]{Conjecture}

\newtheorem{claim}{Claim}
\newtheorem{fact}{Fact}
\newtheorem{case}{Case}

\begin{document}

\title{The spectral radius of graphs without long cycles
}
\author{Jun Gao$^a$, \quad Xinmin Hou$^b$\\
\small $^{a,b}$ Key Laboratory of Wu Wen-Tsun Mathematics\\
\small School of Mathematical Sciences\\
\small University of Science and Technology of China\\
\small Hefei, Anhui 230026, China.\\
}

\date{}

\maketitle

\begin{abstract}
Nikiforov conjectured that for a given integer $k\ge 2$, any graph $G$ of sufficiently large order $n$ with spectral radius $\mu(G)\geq \mu(S_{n,k})$ (or $\mu(G)\ge \mu(S_{n,k}^+))$ contains $C_{2k+1}$ or $C_{2k+2}$(or $C_{2k+2}$), unless $G=S_{n,k}$ (or $G=S_{n,k}^+)$, where $C_\ell$ is a cycle of length $\ell$ and $S_{n,k}=K_k\vee \overline{K_{n-k}}$, the join graph of a complete graph of order $k$ and an empty graph on $n-k$ vertices,  and $S_{n,k}^+$ is the graph obtained from $S_{n,k}$ by adding an edge in the independent set of $S_{n,k}$.
In this paper, a weaker version of Nikiforov's conjecture is considered,
we prove that for a given integer $k\ge 2$, any graph $G$ of sufficiently large order $n$ with spectral radius $\mu(G)\geq \mu(S_{n,k})$ (or $\mu(G)\ge \mu(S_{n,k}^+))$ 
contains a cycle $C_{\ell}$ with $\ell \geq 2k+1$ (or $C_{\ell}$ with $\ell \geq 2k+2$), unless $G=S_{n,k}$ (or $G=S_{n,k}^+)$. These results also imply a result of Nikiforov given in [Theorem 2, The spectral radius of graphs without paths and cycles of specified length, LAA, 2010].


\end{abstract}

\section{Introduction}
In this paper, all graphs considered are simple and finite. Let $G$ be a graph with vertex set $V(G)$ and edge set $E(G)$. For $x\in V(G)$, let $N_G(x)$ be the set of neighbors of $x$ in $G$ and let $N_G[x]=N_G(x)\cup\{x\}$.
In particular, $d_G(x)=|N_G(x)|$ is the degree of $x$ in $G$. For non-empty subset $S\subseteq V$,
let $G[S]$ be the subgraph of $G$ induced by $S$ and write $e_G(S)$ for $e(G[S])$. For disjoint subsets $X$ and $Y$ of $V(G)$, let $E_G(X, Y)$ be the set of edges with one end in $X$ and the other in $Y$ and let $e_G(X, Y)=|E_G(X, Y)|$.
We write $E_G(x,Y)$ instead of $E_G(\{x\}, Y)$ for convenience. All the subscripts defined here will be omitted if $G$ is clear from the context.

 Given a graph $G$,
let $A(G)$ and $\mu(G)$ denote the adjacency matrix and the largest eigenvalue of $A(G)$, called the {\it spectral radius} of $G$.
For a pair of adjacent vertices $u,v$ of $G$, define $$P_v(u)= \{ x \ |\  x\in N(u)\setminus\{v\}, \mbox{ but } x\notin N(v) \}$$ to be the set of private neighbors of $u$ with respect to $v$. Define $G_{u\rightarrow v}$ to be the graph obtained from $G$ by deleting the edges joining $u$ to vertices in $P_v(u)$ and adding new edges connecting $v$ to vertices in $P_v(u)$, that is $V(G_{u\rightarrow v})=V(G)$ and $$E(G_{u\rightarrow v})=(E(G)\setminus E_G(u, P_v(u)))\cup E_{\overline{G}}(v, P_v(u)).$$

As Tur\'{a}n type problems ask for maximum number of edges in graphs of given order not containing a specified family of subgraphs, Brualdi-Solheid-Tur\'{a}n type problems ask for maximum spectral radius of graphs of given order not containing a specified family of subgraphs. A survey about the subject can be found  in~\cite{survy of SEP}.
In this paper, we mainly concern a Brualdi-Solheid-Tur\'{a}n type conjecture proposed by Nikiforov~\cite{SEP of paths and cycles}.    Let $S_{n,k}$ be the graph obtained by joining every vertex of a complete graph of order $k$ to every vertex of an independent set of order $n-k$, that is $S_{n,k}=K_k\vee \overline{K_{n-k}}$, the join graph of $K_k$ and $\overline{K_{n-k}}$, and let $S_{n,k}^+$ be the graph obtained from $S_{n,k}$ by adding a single edge to the independent set of $S_{n,k}$. { For $n=2$ and $k=1$, we set $S_{2,1}=S_{2,1}^+=K_2$.} Write $P_\ell$ and $C_\ell$  for  a path and cycle of order $\ell$ and $P_{\ge \ell}$ and $C_{\ge \ell}$ for a path and cycle of order at least $\ell$.  In~\cite{SEP of paths and cycles}, Nikiforov proved the following theorem and proposed a related conjecture as follows.

\begin{thm}[Nikiforov, 2010]\label{THM: P_2k+2}
Let $k\geq2$, $n>{2^{4k}}$ and $G$ be a graph of order $n$.

(a) If $\mu(G)\geq \mu(S_{n,k})$, then $G$ contains a $P_{2k+2}$ unless $G=S_{n,k}$.

(b) If $\mu(G)\geq \mu(S_{n,k}^+)$, then $G$ contains a $P_{2k+3}$ unless $G=S_{n,k}^+$.
\end{thm}

\begin{conj}[Nikiforov, 2010]\label{CONJ: conjecture 1}
Let $k\geq2$ and $G$ be a graph of sufficiently large order $n$.

(a) If $\mu(G)\geq \mu(S_{n,k})$, then $G$ contains a $C_{2k+1}$ or $C_{2k+2}$ unless $G=S_{n,k}$.

(b) If $\mu(G)\geq \mu(S_{n,k}^+)$, then $G$ contains a $C_{2k+2}$ unless $G=S_{n,k}^+$.

\end{conj}
A first step to attack Conjecture~\ref{CONJ: conjecture 1} was given by Yuan, Wang and Zhai in~\cite{SEP of C5&C6}, they proved that Conjecture~\ref{CONJ: conjecture 1} (a) holds when $k=2$. It seems that it is harder to attack Conjecture~\ref{CONJ: conjecture 1} (b) than to (a), and there is few result known as we have checked.


In this paper,
we prove the following  theorem, which is weaker than Conjecture~\ref{CONJ: conjecture 1}, but is stronger than Theorem~\ref{THM: P_2k+2} (Clearly, Theorem~\ref{THM: P_2k+2} is a corollary of Theorem~\ref{:Main} and we also give a better lower bound for $n$ in Theorem~\ref{:Main}).

\begin{thm}\label{:Main}
Let {$k\ge 2$ and $n\ge {13k^2}$} and let $G$ be a graph of order $n$.

(a) If $\mu(G)\geq \mu(S_{n,k})$, then $G$ contains a $C_{\geq 2k+1}$ unless $G=S_{n,k}$.

(b) If $\mu(G)\geq \mu(S_{n,k}^+)$, then $G$ contains a $C_{\geq 2k+2}$ unless $G=S_{n,k}^+$.
\end{thm}

The rest of the paper is arranged as follows. In Section 2, we give some preliminary facts and lemmas which will be used in the proof of Theorem~\ref{:Main}.
In Section 3, the proof of Theorem~\ref{:Main} is given and in the last section, we discuss some open problems.


\section{Preliminaries}

First we recall some notation not defined in the above section, if $G$ and $H$ are two graphs, we write :
\begin{flushleft}
- $e(G)$ for $|E(G)|$;\\
- $c(G)$ for the circumference (the length of a longest cycle in a graph) of $G$;\\
- $x(G)$ for the eigenvector of $A(G)$ corresponding to the largest eigenvalue $\mu(G)$;\\
- $f(G)=\prod_{u\in V(G)} d(u)$;\\
- $H\subseteq G$ for $H$ being a subgraph of $G$;\\
- $G+H$ for the graph on vertex set $V(G)\cup V(H)$ and $E(G)\cup E(H)$;\\
- $G-X$ for the subgraph of $G$ induced by $V(G)\setminus X$, where $X\subseteq V(G)$.
\end{flushleft}

The following two facts were taken from~\cite{SEP of paths and cycles} and proved in~\cite{Without Path}.
\begin{fact}\label{FACT: EX-P_{2k+2}}
Let $k\geq 1$, $n>3k$ and $G$ be a connected graph of order $n$. If
	\begin{equation}\label{EQN: e1}
	e(G)\geq e(S_{n,k})=kn-(k^2+k)/2
	\end{equation}
then $G$ contains a $P_{2k+2}$, unless there is equality in~(\ref{EQN: e1}) and $G=S_{n,k}$.
\end{fact}

\begin{fact}\label{FACT: EX-P_2k+3}
Let $k\geq 1$, $n>3k$ and $G$ be a connected graph of order $n$. If
	\begin{align}\label{EQN: e2}
	e(G)\geq e(S_{n,k}^{+})=kn-(k^2+k)/2+1
	\end{align}
	then $G$ contains a $P_{2k+3}$, unless there is equality in (\ref{EQN: e2}) and $G=S_{n,k}^{+}$.
\end{fact}

The following classical result was given in~\cite{EG59}.
{
\begin{fact}\label{FACT: P_l}
Let $\ell \geq 1$ and $G$ be  graph of order $n$. If
	\begin{equation*}\label{EQN: e-3}
	e(G)> (\ell-2)n/2,	
\end{equation*}
	then $G$ contains a $P_{\ell}$.
\end{fact}
}

The following fact can be checked directly from the structure of $S_{n,k}$ and $S_{n,k}^+$.
\begin{fact}\label{FACT: P_2k}
Let $n\ge 2k$. If a graph $G$ contains a subgraph $S_{n,k}$ (or $S_{n,k}^+$), then $G$ has a path $P_{\ge 2k-1}$ (or $P_{\ge 2k}$) with two ends in the class $V(K_k)$.
\end{fact}

In the following we give { four} technical lemmas which will be used in the proof of Theorem~\ref{:Main}.

%

\begin{lem}\label{LEM: P_2k+1}
Let $k\geq 1$, $t\ge 2$. If $G$ is a graph with a partition $A\cup B$ of $V(G)$ such that $G[A]\cong K_{t}$, $e(A, B) \ge k|B|$ and $B$ has at least one vertex $b$ with $|N(b)\cap A|>k$, and $|B|>kt$, then $G$ has a path $P_{\ge 2k+1}$ with both ends in $A$ and $|V(P_{\ge 2k+1})\cap B|\ge k$.
	
\end{lem}

\begin{proof}
\begin{claim}\label{CLAIM: c1}
There exist $k$ vertices $b_{1}, b_{2}, \ldots, b_{k}$ in $B$ such that $ |N(b_{1})\cap A|>k$ and $|N(b_{i})\cap A| \geq k$ for $i=2,\ldots, k$.
\end{claim}
Let $S=\{ x\ |\ x\in B \mbox{ and } |N(x)\cap A| \geq k \}.$ Clearly, $b\in S$ and $|N(b)\cap A|>k$. So we can choose $b_1=b$ and it is sufficient to show that $|S|\ge k$. If not then we have
\begin{align*}
k|B|\le e(A, B) \leq |S||A|+(|B|-|S|)(k-1) <kt+|B|(k-1) <k|B|,
\end{align*}
the last inequality holds because $|B|> tk$, the contradiction implies the claim.

\begin{claim}\label{CLAIM: c2}
$G$ contains a path $P_{\geq 2k+1}$ with both ends in $A$ and $\{b_1, \ldots, b_{k}\}\subset V(P_{\ell})$.
\end{claim}
By induction on $k$. For $k=1$, the claim is clearly true. Now suppose $k>1$ and the statement holds for $k-1$.
Choose $a_k$ from $A$ such that $a_kb_k\in E(G)$ and let $A'=A\setminus\{a_k\}$ and $B'=B\setminus\{b_k\}$. Let $G'$ be the subgraph of $G$ induced by $A'\cup B'$. Thus $G'[A']\cong K_{t-1}$,
  $|N_{G'}(b_i)\cap A'|\ge |N_{G}(b_i)\cap A|-1\ge k-1$ for $i=2,\ldots, k-1$, and $|N_{G'}(b_1)\cap A'|\ge |N_{G}(b_1)\cap A|-1>k-1$. By induction hypothesis, $G'$ contains a path $P_{\ell'}$, $\ell'\ge 2k-1$, with two ends, say $a_0$ and $a_{k-1}$, in $A'$ and $\{b_1, \ldots, b_{k-1}\}\subseteq B'$. If $\ell'\ge 2k$ then $P_{\ell'}+a_{k-1}a_k$ is the desired path. Hence assume $\ell'= 2k-1$ and so $|V(P_{\ell'})\cap A'|\le k$. If one of $a_0, a_{k-1}$ belongs to $N_G(b_k)$, say $a_{k-1}b_k\in E(G)$,  then $P_{\ell'}+a_{k-1}b_ka_k$ is a path with length $\ell'+2\ge 2k+1$, as claimed. So, assume none of $a_0, a_{k-1}$ belongs to $N_G(b_k)$. Since $|N_G(b_k)\cap A|\ge k$, there is at least one vertex $a_k'\in N_G(b_k)\cap A$ such that $a_k' \notin V(P_{\ell'})\cap A$. Therefore, $P_{\ell'}+a_{k-1}a_k'b_ka_k$ is a path with length $\ell'+3>2k+1$, as claimed.
This completes the proof of the lemma.
\end{proof}

\noindent{\bf Remark.} In particular, the result is still true if we replace the condition ``$e(A, B) \ge k|B|$ and $B$ has at least one vertex $b$ with $|N(b)\cap A|>k$'' by ``$e(A, B)>k|B|$'' in Lemma~\ref{LEM: P_2k+1}.

\begin{lem}\label{LEM: Move-Neighbor}
Let $G$ be a graph of order $n$ and $uv\in E(G)$. If $P_v(u)\ne \varnothing$ and $P_u(v)\ne \varnothing$
then either $G'=G_{u\rightarrow v}$ or $G'=G_{v\rightarrow u}$ has the property that $c(G')\le c(G)$, $\mu(G')\ge  \mu(G)$, and $f(G')< f(G)$.
\end{lem}
\begin{proof}
Let $x=x(G)$.
Without loss of generality, assume $x_u\leq x_v$.
We show that $G'=G_{u\rightarrow v}$ is the desired graph.
	
 Since $x_v\ge x_u>0$, we have
 \begin{align*}
\mu (G') \geq \frac{x^{T}A(G')x}{x^{T}x}\geq\frac{x^{T}A(G)x}{x^{T}x}=\mu(G).
 \end{align*}
 Since $d_{G}(u)=d_{G'}(u)+|P_v(u)|$, $d_{G}(v)=d_{G'}(u)+|P_u(v)|$, and $d_{G'}(v)=d_{G'}(u)+|P_v(u)|+|P_u(v)|$, we have
 \begin{align*}
\frac{f(G')}{f(G)}=\frac{d_{G'}(u)d_{G'}(v)}{d_{G}(u)d_{G}(v)}< 1,
\end{align*}
that is $f(G')< f(G)$.

 Let $C'= c_{1}c_{2}\ldots c_{\ell}c_1$  be a cycle of length $\ell$ in $G'$.  We show that $G$ has a cycle $C$ of length at least $\ell$. If $v \notin V(C')$ then $C=C'$ also is a cycle in $G$ and we are done.  Now suppose that $v \in V(C')$ and let $a$ and $b$ be the two neighbors of $v$ on $C'$. If $a \notin P_v(u)$ and $b \notin P_v(u)$ then $C=C'$ also is a cycle in $G$ and we are done.
Hence, without loss of generality, assume  $a \in P_v(u)$.
\begin{case}
$b\notin P_u(v)$.
\end{case}
Hence $b\in N_G(u)$. If $u\notin V(C')$, set $C=(C'-\{va, vb\})\cup\{ua, ub\}$, then $C$ is a cycle of length $\ell$ in $G$, as claimed. Hence assume $u\in V(C')$. Let $c$ and $d$ be the two neighbors of $u$ on $C'$. Note that, for all $x\in N_{G'}(u)$, $x\in N_G(v)$. Set $C=(C'-\{va,vb, uc, ud\})\cup\{ua, ub, vc, vd\}$. Thus $C$ is a cycle of length $\ell$ in $G$, as claimed.

\begin{case}
	$b\in P_u(v)$.
\end{case}
If	$u\notin V(C')$, set $C=(C-\{av\})\cup\{au, uv\}$, then $C$ is a cycle of length $\ell+1$ in $G$, as claimed. Hence assume $u\in V(C')$ and let $c$ be the neighbor of $u$ such that $a,v,u,c$ are arranged in clockwise  on $C'$.
Set $C=(C'-\{av, uc\})\cup\{au, vc\}$. Then $C$ is a cycle of length $\ell$ in $G$, as claimed.


The proof is completed.
\end{proof}

\noindent{\bf Remark:} Since $f(G)$ is limited, after finite steps, the moving neighbor operations in Lemma~\ref{LEM: Move-Neighbor} will stop at a graph $H$ with the property that $\mu(H)\ge \mu(G)$, $c(H)\le c(G)$ and, for any edge $uv\in E(H)$, $P_v(u)=\emptyset$ or $P_u(v)=\emptyset$.

\begin{lem}\label{LEM: Au-Matrix}
Given two positive integers $a, b$ and a nonnegative symmetric irreducible matrix $A$ of order $n$, let $\mu$ be the largest eigenvalue of $A$ and let $\mu'$ be the  largest root of the polynomial $f(x)=x^2-ax-b$.
Define  $B=f(A)$ and let $B_j=\sum_{i=1}^nB_{ij}$ for $j=1,2,\ldots, n$. If $B_j\leq0$ for all $j=1,2,\ldots,n$, then $\mu\leq\mu'$ with equality holds if and only if $B_j=0$ for all $j=1,2,\ldots, n$.
\end{lem}
\begin{proof}
Let $\textbf{x}=(x_1,x_2,\dots,x_n)$ be a positive eigenvector of $A$ corresponding to $\mu$ with $\sum_{i=1}^n x_i=1$, and let $\textbf{1}$ be the vector of dimension $n$ with all entries 1.
On one hand,
$$\textbf{1}(B\textbf{x}^T)=\textbf{1}(\mu^2-a\mu-b)\textbf{x}^T=\mu^2-a\mu-b.$$
On the other hand,
$$(\textbf{1}B)\textbf{x}^T=(B_1,B_2,\dots,B_n)\textbf{x}^T=\sum\limits_{i=1}^n B_ix_i\leq0.$$
Hence, we have $\mu^2-a\mu-b\leq0$, which implies that $\mu\leq\mu'$ and the equality holds  {if and} only if  $B_j=0$ for all $j=1,2,\ldots, n$.
\end{proof}

\begin{lem}\label{LEM: k=2}
Given integer $m(\ge 1)$, let $H$ be a graph with components $A_1,\cdots, A_{m-1}$ and let $G=(H+S^+_{t_m,1})\wedge K_1$ and $G'=(H+S^+_{t_m', 1}+S^+_{t_{m+1}', 1})\wedge K_1$. If $t_m'+t_{m+1}'=t_m$ then {$\mu(G)>\mu(G')$.}

\end{lem}
\begin{proof}
The follow claim can be checked directly from the definitions of  eigenvalue and eigenvector.
\begin{claim}\label{CLAIM: L7c1}
Let $G$ is a connected graph and $x$ be a positive eigenvector of $A(G)$ corresponding to $\mu(G)$. For any $uv\in E(G)$, if $P_v(u)=\emptyset$ and $P_u(v)\ne \emptyset$ then $x_v>x_u$, and if $P_v(u)=P_u(v)=\emptyset$ then $x_v=x_u$.
\end{claim}

Let $x=x(G')$ be the eigenvector of $A(G')$ corresponding to $\mu(G')$. Let $u_1$ and $u_2$ be centers of $S_{t_m',1}$ and $S_{t_{m+1}',1}$, respectively. By Claim~\ref{CLAIM: L7c1}, we have $x_{u_1}\ge x_v$ for any $ v\in V(S^+_{t_m', 1})$ and $x_{u_2}\ge x_v$ for any $v\in V(S^+_{t_{m+1}', 1})$.
Without loss of generality, assume $x_{u_1}\ge x_{u_2}$. Note that  $G$ can be seen as the graph obtained from $G'$ by deleting all  the edges of $E(S^+_{t_{m+1}',1})$ and adding new edges connecting $u_1$ to all vertices of $V(S^+_{t_{m+1}',1})$.
Clearly,
$$\mu(G')=xA(G'){x^T}\le xA(G)x^T\le \mu(G), $$
the equality holds only if $x_{u_2}=x_{u_1}$. This implies that $x$ is not an eigenvector of $A(G)$ corresponding to $\mu(G)$, otherwise, we have $x_{u_1}>x_{u_2}$ by Claim~\ref{CLAIM: L7c1}. Therefore, $\mu(G')<\mu(G)$.
\end{proof}	


\section{Proof of Theorem \ref{:Main}}

Clearly, $\mu=\mu(S_{n,k})$ is the largest root of the polynomial
$f(x)=x^2-(k-1)x-k(n-k)$.
Let $G$ be a $C_{\ge 2k+2}$-free graph of order $n$ with maximum spectral radius. 
By the remark of Lemma~\ref{LEM: Move-Neighbor}, we may assume $P_u(v)=\emptyset $ or $P_v(u)=\emptyset$ for any edge $uv\in E(G)$.
For $u\in V(G)$, let $Y_u=V(G)\setminus N[u]$, $S_u=N(u)\cap N(Y_u)$, 
{ and $T_u=N(u)\setminus S_u$}. Let $s_u=|S_u|, t_u=|T_u|$. The following claim holds.

\begin{claim}\label{CLAIM: c3}
 For any $u\in V(G)$, $G[S_u]\cong K_{s_u}$, $K_{s_u}\vee \overline{K_{t_u}}\subseteq G[N(u)]$,
 and $e_{G}(T_u, Y_u)=0$.
\end{claim}
By definition of $T_u$, $e_G(T_u, Y_u)=0$. If there are two vertices $v, w$ in $S_u$ (or a vertex $v\in S_u$ and a vertex $w\in T_u$) such that $vw\notin E(G)$, then there is at least one vertex $x\in Y_u$ with $vx\in E(G)$, so $P_u(v)\not=\emptyset$ (clearly, $x\in P_u(v)$) and $P_v(u)\not=\emptyset$ ( as $w\in P_v(u)$), a contradiction. Hence $G[S_u]\cong K_{s_u}$ and $K_{s_u}\vee \overline{K_{t_u}}\subseteq G[N(u)]$.

By the above claim
and $G$ contains no $C_{\ge 2k+2}$, it is an easy task to check that the following claim holds.
\begin{claim}\label{CLAIM: c5}
For any $u\in V(G)$, we have

(1) $s_u\le 2k$;

(2) $\min\{s_u, t_u\}\le k$; in particular, if $s_u\ge t_u$ then $d(u)\le {2k+1}$.
\end{claim}

\vspace{5pt}
\centerline {\bf Proof of Theorem~\ref{:Main} (a)}
\vspace{5pt}
Suppose $G$ is a $C_{\ge 2k+1}$-free graph with $\mu(G)\ge \mu=\mu(S_{n,k})$ and $G\not=S_{n,k}$.
Let $A=A(G)$, $B=f(A)=A^2-(k-1)A-k(n-k)I$, and $B_u=\sum\limits_{1\leq i\leq n}B_{iu}$ for every $u\in V(G)$.
By definition, for every $u\in V(G)$,
\begin{align}\label{EQN: e3}
B_u=e(N(u), Y_u)+2e(N(u))-(k-2)d(u)-k(n-k).
\end{align}
{By Lemma~\ref{LEM: Au-Matrix}, there must be a vertex $u\in V(G)$ such that $B_u\ge 0$. To get a contradiction, we will show that either $B_u\le 0$  and the equality holds if and only if $G\cong S_{n,k}$ or $B_u>0$ but $\mu(G)<\mu$.

\begin{flushleft}
{\bf Case 1.}  $G[N(u)]$ is connected.
\end{flushleft}

\noindent{\bf Subcase 1.1.} $d(u)>3k$.

By Claim~\ref{CLAIM: c5}, $s_u<t_u$ and hence $\min\{s_u, t_u\}=s_u\le k$. This implies that
\begin{equation}\label{CLAIM: c4}
e(N(u), Y_u)=e(S_u, Y_u)\le k|Y_u|.
\end{equation}
Since $G$ contains no $C_{\ge 2k+1}$, $G[N(u)]$ contains no $P_{\ge 2k}$.  By Fact~\ref{FACT: EX-P_{2k+2}}, we have
\begin{align}\label{EQN: e4}
e(N(u))\leq (k-1)d(u)-(k^2-k)/2,
\end{align}
 with equality if and only if $G[N(u)]\cong S_{d(u), k-1}$.
By (\ref{EQN: e3}),
we have
\begin{align*}
B_u\leq k|Y_u|+2(k-1)d(u)-(k^2-k)-(k-2)d(u)-k(n-k)=0,
\end{align*}
 and the equality holds  if and only if the equalities hold in (\ref{CLAIM: c4}) and (\ref{EQN: e4}) if and only if $s_u=k-1$ and $Y_u=\emptyset$ if and only if $G\cong S_{n,k}$, as desired.


\vspace{5pt}
\noindent{\bf Subcase 1.2.} $d(u)\leq 3k$.

First we claim that
\begin{equation}\label{EQN: e5}
e(S_u, Y_u)>(k-1)|Y_u|.
\end{equation}
If not, then
\begin{equation*}
\begin{split}
B_u&\leq (k-1)|Y_u|+d(u)(d(u)-1)-(k-2)d(u)-k(n-k)\\
&=(d(u)-k+1)^2+k-n\\
&\leq (2k+1)^2+k-n\\
&<0,
\end{split}
\end{equation*}
the last inequality holds because {$n\geq 13k^2$.}

Inequality~(\ref{EQN: e5}) implies that $s_u\ge k$. Note that $s_u\le 2k$. Hence $|Y_u|=n-d(u)-1\ge n-3k-1>9k^2>s_u(k-1)$. By Lemma~\ref{LEM: P_2k+1}, there is a path $P$ of length at least $2k-1$ with two ends in $S_u$ and $|V(P)\cap Y_u|\ge k-1$. If $|S_u\cup T_u|>k$, then it is an easy task to find a cycle of length at least $2k+1$ containing $u$ and $V(P)$, a contradiction. Therefore, $s_u=k$ and $t_u=0$.  Note that $d(u)=s_u=k$, $|Y_u|=n-k-1$, and $e(N(u), Y_u)\le s_u|Y_u|=k|Y_u|$, we have
\begin{align*}
B_u\leq k(n-k-1)+k(k-1)-(k-2)k-k(n-k)=0,
\end{align*}
and the equality holds if and only if $G\cong S_{n,k}$ (otherwise, $E(Y_u)\not=\emptyset$, hence $S_{n-1,k}^+\subseteq G[S_u\cup Y_u]$ and so we can find a path $P_{\ge 2k}$ with two ends in $S_u$ by Fact~\ref{FACT: P_2k}, thus we have a cycle of length $2k+1$ containing $V(P)\cup\{u\}$, a contradiction   ).


\vspace{5pt}
\noindent{\bf Case 2.}   $G[N(u)]$ is disconnected.
\vspace{5pt}

This implies that $S_u=Y_u=\emptyset$ and hence $d(u)=n-1$ and $u$ is the only vertex of $G$ with $G[N(u)]$ being disconnected. Let $A_{1}, \ldots, A_t$ be all the components of $G-\{u\}$ and $n_i=|V(A_i)|$ for $i=1,\ldots, t$.

\vspace{5pt}
\noindent{\bf Subcase 2.1.} There is some $i$ with $n_i>3k$.
\vspace{5pt}

Without loss of generality, assume $n_1>3k$. By Fact~\ref{FACT: EX-P_{2k+2}}, $e(A_1)\le (k-1)n_1-(k^2-k)/2$, and by Fact~\ref{FACT: P_l}, $\sum_{j=2}^t e(A_j)\le (k-1)(n-n_1-1)$. Therefore, we have
\begin{align*}\label{EQN: e6}
e(N(u))=\sum_{i=1}^te(A_i)\leq (k-1)(n-1)-(k^2-k)/2.
\end{align*}
By (\ref{EQN: e3}), we have 
$$B_u\le 2(k-1)(n-1)-(k^2-k)-(k-2)(n-1)-k(n-k)=0.$$
If $B_u<0$ then we are done. Hence assume $B_u=0$. If $\mu(G)=\mu=\mu(S_{n,k})$, then Lemma~\ref{LEM: Au-Matrix} implies that each $B_v=0$ for all $v\in V(G)$. Note that $u$ is the only vertex with $G[N(u)]$ being disconnected. Case 1 implies that if there is a vertex $v$ with $G[N(v)]$ being connected and $B_v=0$, then $G\cong S_{n,k}$, as desired.

\vspace{5pt}
\noindent{\bf Subcase 2.2.}  For all $i\in\{1,\ldots, t\}$, $n_{i}\leq 3k$.

Let $g(x)=x^2-kx-(k-1/2)(n-k)$ and reset $B=g(A)=A^2-kA-(k-1/2)(n-k)I$. Let $B_v=\sum\limits_{1\leq i\leq n}B_{iv}$ for every $v\in V(G)$. Hence
\begin{equation}\label{EQN: e7}
B_v=2e(N(v))+e(N(v), Y_v)-(k-1)d(v)-(k-1/2)(n-k).
\end{equation}
Let $\mu'$ be the largest root of the polynomial $g(x)$. Note that $\mu=\mu(S_{n,k})$ is the largest root of $f(x)=x^2-(k-1)x-k(n-k)$. By simple computation, we have $\mu'<\mu$ when {$n\geq 13k^2$}. By Lemma~\ref{LEM: Au-Matrix}, to prove $\mu(G)\le \mu'(<\mu(S_{n,k}))$, it is sufficient to show $B_v\le 0$ for any $v\in V(G)$.




By Fact~\ref{FACT: P_l}, $e(N(u))\le (k-1)(n-1)$. By (\ref{EQN: e7}) and {$n\ge 13k^2$,} we have
\begin{align*}
B_u\leq 2(k-1)(n-1)-(k-1)(n-1)-(k-\frac 12)(n-k)<0.
\end{align*}

For every $i\in\{1,\ldots, t\}$ and $v\in V(A_i)$,  since $n_i\leq 3k$, we have $d(v)\le 3k$. By Fact~\ref{FACT: P_l}, $e(N(v))\le (k-1)d(v)$.
Note that $u\in N(v)$ and $d(u)=n-1$, we have $e(N(v), Y_v)\le n-d(v)+(d(v)-1)(n_i-d(v))$. By (\ref{EQN: e7}), we have
\begin{equation*}
\begin{split}
B_v\le& 2(k-1)d(v)+n-d(v)+(d(v)-1)(n_i-d(v))\\
& -(k-1)d(v)-(k-1/2)(n-k)\\
<&0,
\end{split}
\end{equation*}
the last inequality holds since { $n\ge 13k^2$} and $d(v)\le n_i\le 3k$.

This  completes the proof of Theorem~\ref{:Main} (a).


\vspace{5pt}
\centerline{ \bf Proof of Theorem~\ref{:Main} (b)}
\vspace{5pt}
Now suppose that there is  a connected $C_{\ge 2k+2}$-free graph $G$ of order $n$ such that $\mu(G)\ge \mu(S^+_{n,k})$ and $G\not= S_{n,k}^+$. Without loss of generality, we assume that $G$ has maximum spectral radius among all of such graphs of order $n$.
Let $A=A(G)$ be the adjacent matrix of $G$, and let $B=f(A)=A^2-(k-1)A-k(n-k)I$ and $B_u=\sum\limits_{1\leq i\leq n}B_{iu}$ for $u\in V(G)$.
By Lemma~\ref{LEM: Au-Matrix}, there must exist some vertex $u\in V(G)$ such that $B_u>0$. In the following, we will find a contradiction, that is we show that either $B_u\le 0$ 
or $B_u>0$ but $\mu(G)\le\mu(S_{n,k}^+)$ with equality if and only if $G\cong S_{n,k}^+$.
\begin{flushleft}
	{\bf Case 1.} 
$G[N(u)]$ is connected.
\end{flushleft}

\noindent{\bf Subcase 1.1.} $d(u)>3k$.

By Claim~\ref{CLAIM: c5}, $\min\{s_u, t_u\}=s_u\le k$. This implies that
\begin{equation*}\label{EQN: e8}
e(N(u), Y_u)=e(S_u, Y_u)\le k|Y_u|.
\end{equation*}
Since $G$ is $C_{\geq 2k+2}$-free, $G[N(u)]$ does not contain a $P_{\ge 2k+1}$. By Fact~\ref{FACT: EX-P_2k+3},
\begin{align*}\label{EQN: e80}
e(N(u))\leq (k-1)d(u)-(k^2-k)/2 +1,
\end{align*}
with equality holds if and only if $G[N(u)]\cong S_{d(u), k-1}^+$. By (\ref{EQN: e3}) and note that $|Y_u|=n-d(u)-1$,
we have
\begin{align*}
B_u\leq k|Y_u|+2(k-1)d(u)-(k^2-k)+2-(k-2)d(u)-k(n-k)=2,
\end{align*}
where the equality holds if and only if $Y_u=\emptyset$ (otherwise, $s_u=k$, contradicts to $G[N(u)]\cong S^+_{d(u),k-1}$) and hence $G\cong S_{n,k}^+$; and $B_u=1$ if and only if $e(N(u), Y_u)=e(S_u, Y_u)=k|Y_u|-1$ and $G[N(u)]\cong S^+_{d(u), k-1}$ if and only if $|S_u|=k-1$ and $|Y_u|=1$, that is $G$ is a subgraph of $S_{n,k}^+$ with $e(G)<e(S_{n,k}^+)$ and so we have $\mu(G)<\mu(S_{n,k}^+)$.

\vspace{5pt}
\noindent{\bf Subcase 1.2.} $d(u)\leq 3k$.
\vspace{5pt}

With a same argument as the proof of inequality (\ref{EQN: e5}),  we have $$e(N(u), Y_u)=e(S_u, Y_u)>(k-1)|Y_u|.$$
So $s_u\ge k$. Since $s_u\le 2k$, $|Y_u|=n-d(u)-1>9k^2>s_uk>s_u(k-1)$. By Lemma~\ref{LEM: P_2k+1}, there is a path $P_{\ge 2k-1}$ with two ends in $S_u$ and $|V(P_{\ge 2k-1})\cap Y_u|\ge k-1$. If $|S_u\cup T_u|>k+1$, then it is an easy task to find a cycle of length at least $2k+2$ containing $u$ and $V(P_{\ge 2k-1})$, a contradiction. Therefore, $s_u+t_u\le k+1$. 

If $s_u=k$ and $t_u=0$, then  $e(N(u))=k(k-1)/2$. By (\ref{EQN: e3}), we have
\begin{align*}
B_u\leq k(n-k-1)+k(k-1)-(k-2)k-k(n-k)=0.
\end{align*}

If $s_u=k$ and $t_u=1$, then $e(N(u))=k(k+1)/2$. By (\ref{EQN: e3}), we have
\begin{align*}
B_u\leq k(n-k-2)+k(k+1)-(k-2)(k+1)-k(n-k)=2,
\end{align*}
and equality holds if and only if $e(S_u, Y_u)=k|Y_u|$ if and only if $G\cong S_{n,k}^+$ (Otherwise, $E(G[Y_u])\not=\emptyset$ and so $G[S_u\cup Y_u]$ contains a subgraph $S_{n-2,k}^+$. By Fact~\ref{FACT: P_2k},  $G[S_u\cup Y_u]$ has a path $P_{\ge 2k}$ with two ends in $S_u$ and so it is an easy task to find a cycle $C_{\ge 2k+2}$ containing $V(P_{\ge 2k})\cup\{u\}\cup T_u$, a contradiction), 
as desired; and $B_u=1$ if and only if $e(S_u, Y_u)=k|Y_u|-1$ if and only if $G\subseteq S_{n,k}^+$ with $e(G)<e(S_{n,k}^+)$ (Otherwise, $G[S_u\cup Y_u]$ contains a subgraph $S_{n-2, k}^+-e$, where $e$ is some edge joining a vertex in $S_u$ to a vertex in $Y_u$, similar as in Fact~\ref{FACT: P_2k}, we can find a path $P_{\ge 2k}$ in $G[S_u\cup Y_u]$ with two ends in $S_u$ and so again we get a contradiction), therefore $\mu(G)<\mu(S_{n,k}^+)$, as claimed.

Now suppose $s_u=k+1$ and $t_u=0$. Then $e(N(u))=k(k+1)/2$.
If $e(S_u, Y_u)<k|Y_u|-1$, then (\ref{EQN: e3}) implies that
$$B_u\leq k(n-k-2)-2+k(k+1)-(k-2)(k+1)-k(n-k)=0.$$
Hence assume  $e(S_u, Y_u)\geq k|Y_u|-1$. We first claim that there is no vertex $y$ in $Y_u$ such that $e(y, S_u)=k+1$. If not, choose a vertex $y\in Y_u$ such that $e(y, S_u)$ is minimal among all vertices in $Y_u$, then $e(y, S_u)\le k-1$ because $e(S_u, Y_u)\le k|Y_u|$. Hence
 $$e(S_u, Y_u\setminus\{y\})=e(S_u, Y_u)-e(y, S_u)\ge k|Y_u|-1-(k-1)=k|Y_u\setminus\{y\}|.$$
Note that $|Y_u\setminus\{y\}|=n-k-3>(k+1)k$. By Lemma~\ref{LEM: P_2k+1}, $G[S_u\cup(Y_u\setminus\{y\})]$ contains a path $P_{\ge 2k+1}$ with two ends in $S_u$ and hence, combining the vertex $u$, we get a cycle of length at least $2k+2$ in $G$, a contradiction. By the claim and $k|Y_u|-1\le e(S_u, Y_u)\le k|Y_u|$, $Y_u$ contains at most one vertex $y$ with $e(y, S_u)=k-1$.
Now choose a vertex $v$ from $S_u$ such that $d(v)$ is minimal among all vertices in $S_u$. Then $N(v)\cap Y_u\subseteq N(w)\cap Y_u$ for any $w\in S_u$ by the minimality of $d(v)$ and the assumption that $P_v(w)=\emptyset$ or $P_w(v)=\emptyset$. Let $S_u'=S_u\setminus\{v\}$ and $Y_u'=Y_u\setminus(N(v)\cap Y_u)$. We claim that $$|Y_u'|\geq k+1.$$
If not, we have
$$e(S_u, Y_u)\ge (k+1)(|Y_u|-k)=k|Y_u|+|Y_u|-k(k+1)>k|Y_u|,$$
a contradiction. Therefore, $G[S_u'\cup Y_u']$ contains a subgraph isomorphic to $S_{2k, k}$. By Fact~\ref{FACT: P_2k},  $G[S_u'\cup Y_u']$ contains a path $P_{\ge 2k-1}$ with two ends, say $\{a, b\}$, in $S_u'$. Choose a vertex $z$ from $N(v)\cap Y_u$ so that $zb\in E(G)$, this can be done since $N(v)\cap Y_u\subseteq N(b)\cap Y_u$.   So $P_{\ge 2k-1}+bzvua$ is a cycle of length at least $2k+2$ in $G$, a contradiction.

\begin{flushleft}
	{\bf Case 2.} $G[N(u)]$ is not connected.
\end{flushleft}
This implies that $S_u=Y_u=\emptyset$ and hence $d(u)=n-1$. 
Let $A_{1}, \ldots, A_t$ be the components of $G-\{u\}$ and let $n_i=|V(A_i)|$.
  Without loss of generality, assume $n_1\le \ldots\le n_t$ and let $s$ be the largest integer such that $n_s\le 3k$.
Set $H=A_1+\cdots+ A_{s}$ and $H'=A_{s+1}+\cdots+ A_{t}$. Let $h=|V(H)|$. So $|V(H')|=n-1-h$.

If $V(H')\not=\emptyset$, since $n_{i}>3k$ for $i\in\{s+1, t\}$ and $G$ is $C_{\ge 2k+2}$-free, Fact~\ref{FACT: EX-P_2k+3} implies that
\begin{equation}\label{EQN: e10}
\begin{split}
e(H')&=\sum_{i=s+1}^t e(A_i)\\
&\le \sum_{i=s+1}^t \left[(k-1)n_i-(k^2-k)/2+1\right]\\
&\le (k-1)|V(H')|-(k^2-k)/2+1,
\end{split}
\end{equation}
{with equality  holds if and only if $k=2$ and $A_i\cong S^+_{n_i, 1}$ for each $i=\{s+1,\ldots, t\}$ or $k\ge 3$ and $H'$ has only one component.

If $|V(H)|<2k-2$ and $V(H)\not=\emptyset$,
then $$e(H)\le |V(H)|(|V(H)|-1)/2< (k-1)h-1.$$
{
If $V(H)=\emptyset$ and $k\ge 3$ or $k=2$ and there exists one component $A_i\not=S^+_{n_i,k-1}$ for some $i\in\{s+1, \ldots, t\}$,
then $e(H')\le (k-1)|V(H')|-(k^2-k)/2$ because $G[N(u)]$ is disconnected.
} Therefore, we always have
$$2e(N(u))=2e(H)+2e(H')\le 2(k-1)(n-1)-(k^2-k).$$
Note that $|V(H)|<2k-2$ implies that $V(H')\not=\emptyset$.
By (\ref{EQN: e3}),
\begin{equation*}
\begin{split}
B_u
&\le 2(k-1)(n-1)-(k^2-k)-(k-2)(n-1)-k(n-k)=0.
\end{split}
\end{equation*}
{ If $V(H)=\emptyset$, $k=2$ and each component $A_i\cong S^+_{n_i, 1}$,  then $G=H'\wedge K_1$. By Lemma~\ref{LEM: k=2}, we have $\mu (G)< \mu(S_{n,k}^+)$.}


Therefore,  in the following we assume $|V(H)|\ge 2k-2$. Recall that $\mu=\mu(S_{n,k})$.

 We claim that $-1/3\le\frac{x-\mu}{n-1-\mu}\le1$ for $x\in \{0,1,\ldots, n-1\}$. When $x=n-1$, we have $\frac{x-\mu}{n-1-\mu}=1$. So it is sufficient to show that $\frac{\mu}{n-1-\mu}\le 1/3$.
Since {$\mu=(k-1)/{2}+\sqrt{k(n-k)+(k-1)^2/4}\le (k-1)/{2}+\sqrt{kn} $ and $n\ge 13k^2$}, we have $n \ge \sqrt{n}\sqrt{13k^2}\ge 5\sqrt{nk}> 4\sqrt{nk}+2k\ge 4\mu+1$. This implies that $\frac{\mu}{n-1-\mu}\le 1/3$.

\vspace{5pt}
\noindent{\bf Subcase 2.1.} $|V(H')|=\emptyset$.
\vspace{5pt}

{
Let $$g(x)=x^2-(k-1)x-k(n-k)-\frac{x-\mu}{n-1-\mu}(n+k^2-k-1)$$ and reset $$ B=g(A)=A^2-(k-1)A-k(n-k)I-\frac{n+k^2-k-1}{n-1-\mu}(A-\mu I).$$
Let $B_v=\sum\limits_{1\leq i\leq n}B_{iv}$ for each $v\in V(G)$. Hence
\begin{equation}\label{EQN: e12}
B_v=2e(N(v))+e(N(v), Y_v)-(k-2)d(v)-k(n-k)-\frac{d(v)-\mu}{n-1-\mu}(n+k^2-k-1).
\end{equation}
Clearly, $\mu=\mu(S_{n,k})$ is still the largest root of $g(x)$. By Lemma~\ref{LEM: Au-Matrix}, to show $\mu(G)<\mu(S_{n,k}^+)$, it is sufficient to show $B_v\le 0$ for any $v\in V(G)$ as $\mu(S_{n,k})<\mu(S_{n,k}^+)$.
Since $G$ is $C_{\ge 2k+2}$-free, $H$ is $P_{\ge 2k+1}$-free. By Fact~\ref{FACT: P_l}, $e(H)\le(2k-1)h/2$.
By (\ref{EQN: e12}), we have
\begin{equation*}
\begin{split}
B_u=&2e(H)-(k-2)(n-1)-k(n-k)-n-k^2+k+1\\
{\le}&(k+1)(n-1)-k(n-k)-n-k^2+k+1\\
=&0.
\end{split}
\end{equation*}

For any $v\in V(G)$ with $v\ne u$, note that $G[N(v)]$ is connected and $d(v)\le 3k$ because each component of $H$ has order at most $3k$.  By Fact~\ref{FACT: P_l}, $e(N(v))\le (2k-1)d(v)/2$. Suppose $v\in V(A_i)$ then $Y_v=(\cup_{j\not=i}V(A_j))\cup (Y_v\cap V(A_i))$. Thus
$$e(N(v), Y_v)
\le n-d(v)+(d(v)-1)(n_i-d(v)).$$
By (\ref{EQN: e12}), we have
\begin{equation*}
\begin{split}
B_v\le& (2k-1)d(v)+n-d(v)+(d(v)-1)(n_i-d(v))\\
&-(k-2)d(v)-k(n-k)+(n+k^2-k-1)/3\\
\le& kd(v)+n+(n_i-1)^2/4-k(n-k)+(n+k^2-k-1)/3\\
\le& {3k^2+(3k-1)^2/4+(k^2-k-1)/3+(4/3-k)n }  \\
\le&0
\end{split}
\end{equation*}
where the first inequality holds since  $\frac{x-\mu}{n-1-\mu}\ge -1/3$ for any $x\in \{0,1,\ldots, n-1\}$, and the last inequality holds since {$n\ge 13k^2, k\ge 2$ and $d(v)\le n_i\le 3k$.}

\vspace{5pt}
\noindent{\bf Subcase 2.2.} $|V(H')|\not=\emptyset$.
\vspace{5pt}

Let
 $$g(x)=x^2-(k-1)x-k(n-k)-\frac{x-\mu}{n-1-\mu}(h+2)$$ and reset $$ B=g(A)=A^2-(k-1)A-k(n-k)I-\frac{h+2}{n-1-\mu}(A-\mu I).$$
Let $B_v=\sum\limits_{1\leq i\leq n}B_{iv}$ for each $v\in V(G)$. Hence
 \begin{equation}\label{EQN: e9}
 B_v=2e(N(v))+e(N(v), Y_v)-(k-2)d(v)-k(n-k)-\frac{d(v)-\mu}{n-1-\mu}(h+2).
 \end{equation}
Clearly,
$\mu=\mu(S_{n,k})$ is still the largest root of $g(x)$. Thus,
 to show $\mu(G)< \mu(S_{n,k}^+)$, it is sufficient to show $B_v\le 0$ for any $v\in V(G)$ with a same reason as in Subcase 2.1.

Since $G$ is $C_{\ge 2k+2}$-free, $H$ is $P_{\ge 2k+1}$-free. By Fact~\ref{FACT: P_l}, $e(H)\le(2k-1)h/2$.
Hence
\begin{equation*}
\begin{split}
2e(H)+2e(H')\le &(2k-1)h+2(k-1)(n-h-1)-(k^2-k)+2\\
             =& 2(k-1)(n-1)+h-(k^2-k)+2.
\end{split}
\end{equation*}
By (\ref{EQN: e9}), we have
\begin{equation*}
\begin{split}
B_u=&2e(H)+2e(H')-(k-2)(n-1)-k(n-k)-h-2\\
=&k(n-1)-(k^2-k)-k(n-k)\\
=&0.
\end{split}
\end{equation*}

For any $v\in V(G)$ with $v\ne u$, we have that $G[N(v)]$ is connected. If $v \in V(H)$ then $d(v)\le 3k$ because each component of $H$ has order at most $3k$.  By Fact~\ref{FACT: P_l}, $e(N(v))\le (2k-1)d(v)/2$. If $v\in V(A_i)$ for some $i\in\{1,\ldots, s\}$, note that $Y_v=(\cup_{j\not=i}V(A_j))\cup (Y_v\cap V(A_i))$, then
$$e(N(v), Y_v)
\le n-d(v)+(d(v)-1)(n_i-d(v)).$$
By (\ref{EQN: e9}), we have
\begin{equation*}
\begin{split}
B_v\le& (2k-1)d(v)+n-d(v)+(d(v)-1)(n_i-d(v))\\
&-(k-2)d(v)-k(n-k)+(h+2)/3\\
\le& kd(v)+n+(n_i-1)^2/4-k(n-k)+(h+2)/3\\
\le& { 3k^2+(3k-1)^2/4+k^2+2/3+(4/3-k)n }\\
\le&0
\end{split}
\end{equation*}
where the first inequality holds since  $\frac{x-\mu}{n-1-\mu}\ge -1/3$ for any $x\in \{0,1,\ldots, n-1\}$, and the last inequality holds since {$n\ge 13k^2$,  $k\ge 2$, $d(v)\le n_i\le 3k$ and $h\le n-1$.}


Now assume $v\in V(H')$.
\begin{claim}
We have $k\ge 3$.
\end{claim}
If not, then $k=2$. We first claim that each component $A_i$ of $H'$ has maximum degree $n_i-1$. Otherwise, choose a vertex $x$ with maximum degree in $A_i$ and a vertex $y\in N_{A_i}(x)$ such that $y$ has a neighbor $z\notin N_{A_i}[x]$, such a vertex exists since $d_{A_i}(x)<n_i-1$. By the maximality of the degree of $x$, there is at least one vertex $z'\in N_{A_i}(x)$ but $z'\notin N_{A_i}[y]$. This implies that $P_y(x)\not=\emptyset$ and $P_x(y)\not=\emptyset$, a contradiction. Hence $A_i$ is a subgraph of $S^+_{n_i,1}$, otherwise, $A_i$ has a path $P_{\ge 5}$ and hence $G$ has a cycle $C_{\ge 6}$, a contradiction.
By the maximality of the spectral radius of $G$ and Lemma~\ref{LEM: k=2}, we may assume $H'=A_t \cong S_{n_t,1}^+$. If $h\ge 4$, then, for any $v\in V(H')$, we have
\begin{equation*}
\begin{split}
B_v&=2e(N(v))+e(N(v), Y_v)-(2-2)d(v)-2(n-2)+(h+2)/3\\
&\le {2(n-1-h)+h-2(n-2)+(h+2)/3}\\
&\le 0.
\end{split}
\end{equation*}
If $h\le3$, again by the maximality of radius of $G$, we may assume $H\cong K_h$. Since $K_h\cong S_{h,1}^+$, we have $G\cong S_{h,1}^++S_{n_t, 1}^+$. By Lemma~\ref{LEM: k=2}, $\mu(G)<\mu(S_{n,k}^+)$, a contradiction. The claim holds.

\vskip 5pt
In the following, we assume $k\ge 3$.
If $d(v){ \le}3k$, we claim that $e(N(v),Y_v)> (k-1)Y_v$. If not, then $e(N(v),Y_v)\le (k-1)|Y_v|$. By (\ref{EQN: e9}),
\begin{equation*}
\begin{split}
B_v&\le d(v)(d(v)-1)+(k-1)|Y_v|-(k-2)d(v)-k(n-k)+(h+2)/3\\
&\le {(2k+1)^2+k-n+(n+1)/3}\\
&\le 0
\end{split}
\end{equation*}
the inequality holds because {$n\ge 13k^2$}, a contradiction. The claim is true. Therefore, $e(N(v),Y_v)> (k-1)|Y_v|$. With a same argument as in Subcase 1.2, we have $d(v)=s_v+t_v\le k+1$.
Since $G$ is $C_{\ge 2k+2}$-free,  there are at most $k$ points $x$ in $Y_v\cap V(H')$ such that $e(x, N(v))=k+1$. Hence $$e(S_v,Y_v)=e(S_v,Y_v\cap V(H'))+e(u, V(H))\le k(n-1-h-d(v))+k+h.$$
By (\ref{EQN: e9}),
\begin{equation*}
\begin{split}
B_v\leq &d(v)(d(v)-1)+k(n-1-h-d(v))+h+k\\
&-(k-2)d(v)-k(n-k)+(h+2)/3\\
=& {8/3+k+(4/3-k)h}\\
\leq& 0,
\end{split}
\end{equation*}
the last ineqaulity holds since {$h\ge 2k-2$ and $k\ge 3$}.

If $d(v){>} 3k$, then  Fact~\ref{FACT: EX-P_2k+3} implies that $e(N(v))\le d(v)(k-1)-(k^2-k)/2+1$ because $G$ is $C_{\ge 2k+2}$-free. By Claim~\ref{CLAIM: c4}, we have $s_v\le k$.
Hence $$e(S_v,Y_v)=e(S_v, Y_v\cap V(H'))+e(u, V(H))\le k(n-1-h-d(v))+h.$$
By (\ref{EQN: e9}), we have
\begin{equation*}
\begin{split}
B_v\leq &2d(v)(k-1)-k^2+k+2+k(n-1-h-d(v))+h\\
&-(k-2)d(v)-k(n-k)+(h+2)/3\\
=&8/3+4h/3-kh\\
\leq& 0,
\end{split}
\end{equation*}
the last inequality holds since {$h\ge 2k-2$ and $k\ge 3$}.







This  complets the proof of Theorem~\ref{:Main} (b).

\section{Concluding remarks}

In this paper, we prove a stronger version of Theorem~\ref{THM: P_2k+2} and a weak version of Conjecture~\ref{CONJ: conjecture 1}. We believe that the following result  also is true.
\begin{conj}\label{c2}
Let $k\geq2$, $L\ge 0$ and let $G$ be a graph of sufficiently large order $n$.

(a) If $\mu(G)\geq \mu(S_{n,k})$, then $G$ contains a $C_{\ell}$ with $\ell\in \{2k+1, 2k+2, \ldots, 2k+2+L\}$ unless $G=S_{n,k}$.

(b) If $\mu(G)\geq \mu(S_{n,k}^+)$, then $G$ contains a $C_{\ell}$ with $\ell\in \{2k+2, \ldots, 2k+2+L\}$ unless $G=S_{n,k}^+$.
\end{conj}
Theorem~\ref{:Main} states that Conjecture~\ref{c2} holds for $L=n-2k-2$ and Conjecture~\ref{CONJ: conjecture 1} hopes that Conjecture~\ref{c2} holds for $L=0$. 


\noindent {\bf Acknowledgement. } The research is
supported by National Natural Science Foundation of China (no.  11671376) and National Natural Science Foundation of Anhui Province (no. 1708085MA18).

\end{document}